\documentclass[12pt,a4paper,reqno]{amsart}
\usepackage[T2A]{fontenc}

\usepackage{hyphenat}

\usepackage[english]{babel}

\usepackage{xcolor}
\usepackage{hyperref}

\usepackage{amsmath, amssymb, amsthm, mathtools, thmtools, mathdots, graphicx}
\usepackage{tikz}
\usepackage{multicol}
\usepackage[shortlabels]{enumitem}

\definecolor{linkcolor}{HTML}{970505}
\definecolor{urlcolor}{HTML}{799B03} 
\definecolor{citecolor}{HTML}{799B03}

\hypersetup{pdfstartview=FitH,  linkcolor=linkcolor,urlcolor=urlcolor, citecolor = citecolor,  colorlinks=true }

\DeclareMathOperator{\grad}{grad}

\usepackage{graphicx}
\usepackage[all]{xy}
\usepackage[left=3cm,right=3cm,top=3cm,bottom=3cm,bindingoffset=0cm]{geometry}
1

\usepackage{subfig}

\theoremstyle{definition}
\newtheorem{theorem}{Theorem}[section]
\newtheorem{proposition}[theorem]{Proposition}
\newtheorem{lemma}[theorem]{Lemma}
\newtheorem{corollary}[theorem]{Corollary}

\theoremstyle{definition}
\newtheorem{example}[theorem]{Example}
\newtheorem{remark}[theorem]{Remark}

\theoremstyle{definition}
\newtheorem{definition}[theorem]{Definition}
\usepackage{import}
\usepackage{xifthen}
\usepackage{pdfpages}
\usepackage{transparent}

\begin{document}

\title{On Courant-like bound for Neumann domain count}

\author[1]{Aleksei Kislitsyn}

\address{Department of Higher Geometry and Topology, Faculty
of Mathematics and Mechanics, Moscow State University, Leninskie Gory,
GSP-1, 119991, Moscow, Russia}
\address{Independent University of Moscow, Bolshoy
Vlasyevskiy Pereulok 11, 119002, Moscow, Russia}
\email{aleksei.kislitcyn@math.msu.ru}

\begin{abstract}
    In this work we show that in general there is no Courant-like bound  for Neumann domain count. In order to do that we construct a sequence of domains $\Omega^n$ such that the first Dirichlet eigenfunction for $\Omega^n$ has at least $n$ Neumann domains. Also a special case of convex domains is considered and sufficient conditions for existence of Courant-like bound for small eigenvalues are found. 
\end{abstract}

\maketitle

\section{Introduction}

Let us consider the Dirichlet eigenvalue problem for the Laplace operator $\Delta := \frac{\partial^2}{\partial x^2} + \frac{\partial^2}{\partial y^2}$ in a bounded domain $\Omega$,

	\begin{equation}
    \label{pr1}
 \begin{cases}
   -\Delta u = \lambda u,
   \\
   u|_{\Gamma^D} = 0,

 \end{cases}
\end{equation}
here $\Gamma^D = \partial \Omega$. It is well know that this problem has the discrete spectrum  $0 \leq \lambda_1(\Omega) \leq \lambda_2(\Omega) \leq ... \leq \lambda_k(\Omega) \leq ...,$ see e. g. \cite{LMP}. By $u_k(\Omega)$ we denote an eigenfunction corresponding to the eigenvalue   $\lambda_k(\Omega)$. Also we write just $\lambda_k$ and $u_k$ instead of  $\lambda_k(\Omega)$ and $u_k(\Omega)$ if the domain is fixed.

Nodal geometry is a classical branch of spectral geometry. It studies properties of the partition into nodal domains (see Definition \ref{def_nod_line} below). Courant nodal domain theorem is a very important result in nodal geometry. It states that the function $u_k$ has at most $k$ nodal domains. For more information on nodal geometry, see e.g. \cite[\S 4]{LMP}.

On the boundary of each nodal domain eigenfunction of the problem (\ref{pr1}) is equal to zero, so it satisfies the Dirichlet condition. Therefore, a natural question arises: is it possible to similarly “cut” a region into subregions so that on the boundary of each subregion the function satisfies the Neumann condition? That leads to a concept of Neumann domains. The study of this concept was started in  \cite{MR3087922, MR3151083}. 

Let us remark that there is no general definition of the Neumann domains for eigenfunctions of the problem (\ref{pr1}). The Neumann domains are  defined in special cases, when eigenfunctions are Morse functions in \cite{MR3535866} and  when eigenfunctions are analytic functions in \cite{anoop2024neumanndomainsplanaranalytic}. The definition of the Neumann domains is quite long, so we give it below, see Definition \ref{def_nd}. In papers \cite{MR3535866, MR4244877, MR4668090} some properties of Neumann domains are studied in the Morse case.

In this work we study an existence of Courant-like bounds for a Neumann domain count  in analytic case. By the  Courant-like bound we understand an estimate of the number of the  Neumann domains in terms of an eigenvalue number. This question was formulated in  \cite{MR3535866}. We give a negative answer to it. Let us formulate the following auxiliary result.

\begin{theorem}

\label{th_uniq}
    Let $u$ be an analytic solution of the problem (\ref{pr1}) and $\Omega_i$ be a Neumann domain of the function $u$. Let $x_1, \, x_2 \in \overline{\Omega}_i$ be local maxima of the function $u$. Then there exists a curve of local maxima $\theta$ such that $x_1, \, x_2 \in \theta$.

\end{theorem}

\begin{remark}
    A structure of the local maxima set of an eigenfunction is described in \cite[ \S 2.1, \S 2.2]{anoop2024neumanndomainsplanaranalytic}. In the Dirichlet case the set of local maxima consists of isolated local maxima and simple closed curves of local maxima. Both of them are isolated in the set of all critical points

\end{remark}

We give a proof of the Theorem \ref{th_uniq} in  a Section \ref{sect_proof_th}. Then we have following corollaries from this Theorem.

\begin{corollary}
\label{col_1}
    Under the conditions of the Theorem \ref{th_uniq}, the number of the Neumann domains of the function $u$ is greater or equal than the sum of the number of isolated local maxima and the number of curves of local maxima.
\end{corollary}

\begin{corollary}
\label{col_2}
    Under the conditions of the Theorem \ref{th_uniq}, twice the number of the Neumann domains of the function $u$ is greater or equal than the number of its nodal domains.
\end{corollary}

In a Section \ref{examples} we construct a sequence of domains $\Omega^n$ such that $u_1(\Omega^n)$ has at least $n$ local maxima. Moreover these maxima pairwise don not lie on a same curve of critical points. According to the Corollary \ref{col_1} we have that for any given $n \in \mathbb{N}$ there exists a domain such that its first Dirichlet eigenfunction has at least $n$ Neumann domains. Hence, in general Courant-like bound
for Neumann domain count is impossible

Now let us consider convex domains and find sufficient conditions under which a Courant-like estimate is possible.

\begin{theorem}
\label{th1}
Let $\Omega$ be a planar Euclidean domain with an analytic boundary of positive curvature. Then $u_1$ has exactly one Neumann domain and it is homeomorphic to a disk without a point.

\end{theorem}

Note that the convexity and analyticity  conditions  are essential. We provide corresponding examples in the Section  \ref{examples}.

By extremum we understand a maximum or a minimum point. Note that $u_2$ has at least two extrema in the  domain $\Omega$, see Remark \ref{rem_two_extr}.

\begin{theorem}
\label{th2}
    Let $\Omega$ be a planar convex Euclidean domain with analytic boundary then
    \begin{enumerate}
        \item  the function $u_2$ has at least three Neumann domains, 
        \item if the number of extrema of the function $u_2$ in $\Omega$ is equal to $2$, then the function $u_2$ has exactly three Neumann domains.
    \end{enumerate}
    
\end{theorem}

We provide a proof of the Theorem \ref{th1} and Theorem \ref{th2} in the Section \ref{sect_proof_th}. A question  when $u_2$ has exactly two extrema in a convex domain is quite non-trivial and was studied in \cite{DEREGIBUS2022109496}.

\begin{figure}
    \centering

    \centering
    \subfloat[\centering ]{{\includegraphics[width=6cm]{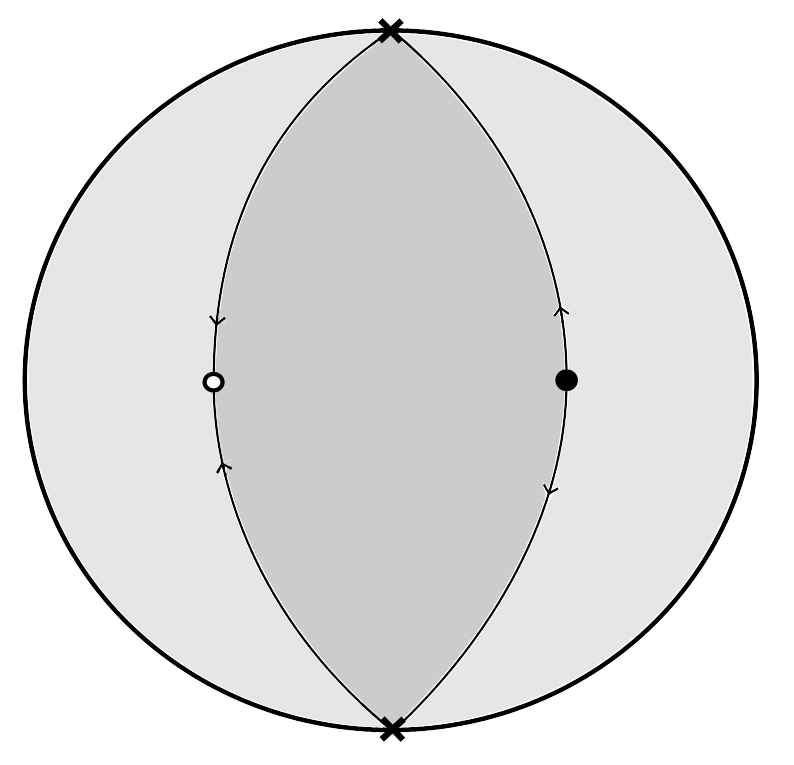} }}%
    \qquad
    \subfloat[\centering ]{{\includegraphics[width=6cm]{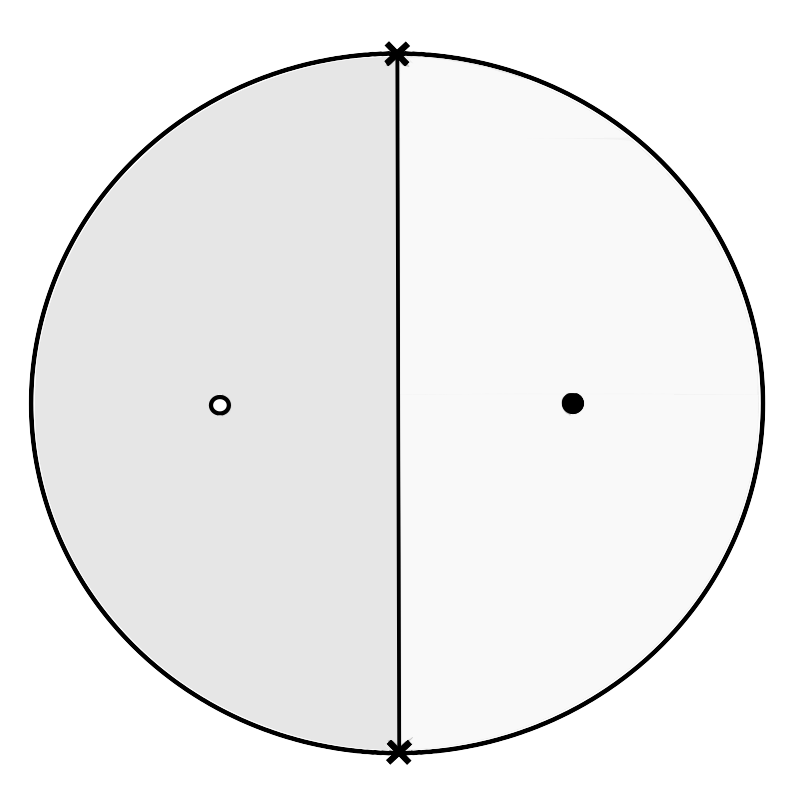} }}
    \caption{The Neumann domains (A) and the nodal domains (B) of the second eigenfunction of the Dirichlet problem for the unit circle. Notations: black circle is a local maximum, white circle is a local minimum, cross is a saddle point.}%
    \label{fig:example}%

    \label{fig1}
\end{figure}

\section{Neumann domains}
\label{sec_def}
\subsection{Definition of the Neumann domains}

Let $u$ be an eigenfunction of the problem (\ref{pr1}). Suppose that there exists the analytic extension of $u$ to $\Omega' \supset \Omega $. Then we say that a point $x \in \partial \Omega$ is a local maximum, minimum, or saddle point if it is such for the extension of the function $u$ to $\Omega'$. Let us clarify that by a local maximum or minimum we mean a non-strict local maximum or minimum. We say that a critical point is a saddle point if it is neither a maximum nor a minimum. A saddle point can be degenerate. Let us introduce some notations. Let
$$\mathcal{C} = \{ x \in \overline{\Omega}: \grad u (x) = 0  \},$$
$$\mathcal{M}_- = \{ q \in \mathcal{C}: q \, \,\text{local minimum} \, \, u \}, $$
$$\mathcal{M}_+ = \{ p \in \mathcal{C}: p \, \,\text{local maximum} \, \, u \}, $$
$$\mathcal{S} = \{r \in \mathcal{C}: r \, \,\text{saddle point} \, \, u  \},$$
$$\mathcal{C}_{sing}  = \{x \in \mathcal{C}: u(x) = 0  \}.$$

For a point $x \in \Omega$ by $\gamma_x(t)$ we denote the integral curve of the vector $-\grad u$ such that $\gamma_x(0) = x$. According to \cite[Lemma 2.23]{anoop2024neumanndomainsplanaranalytic}, as $t$ increases, one of the following two behaviors of the curve $\gamma_x(t)$ is possible.
\begin{itemize}
    \item Curve $\gamma_x(t)$ reaches $\Gamma^D$ in a finite time. That means that there exists $t_x$ such that, $x_0 := \gamma_x(t_x)  \in \Gamma^D$.
    \item Curve $\gamma_x(t)$ tends to a critical point in $\overline{\Omega}$ when $t \to + \infty$. That means that there exists  $x_0 :=\displaystyle{ \lim_{t \rightarrow  +\infty}} \gamma_x(t)$ and $x_0 \in \mathcal{C}$. 
\end{itemize}

In any of these cases, we say that the point $x_0$ the \textit{right end} of the curve $\gamma_x(t)$. Similarly, we can consider the case of decreasing $t$ and define the \textit{left end} of the curve $\gamma_x(t)$. In the case when there exists $\displaystyle{ \lim_{t \rightarrow \pm \infty}} \gamma_x(t)$, we will denote it by $\gamma_x(\pm \infty).$

\begin{definition}
Let  $x_0 \in \overline{\Omega}$ be a critical point of the function $u$, then sets 
$$W^s(x_0) := \{ x \in \overline{ \Omega} : \gamma_x(+ \infty) = x_0  \},$$
$$W^u(x_0) := \{ x \in \overline{ \Omega} : \gamma_x(- \infty) = x_0  \},$$
are called \textit{stable} and \textit{unstable sets} of the  point $x_0$.
\end{definition}

For a point $r \in \mathcal{S}$  we define $L(r) := W^s(r) \cup W^u(r)$ and
$$L := \bigcup_{r \in \mathcal{S}} \overline{L(r)},$$
here closure is taken in $\mathbb{R}^2$. In the paper \cite{anoop2024neumanndomainsplanaranalytic}, several equivalent definitions of the Neumann domains are proposed, see \cite[Definition 3.5, Definition 3.6, Remark 3.7]{anoop2024neumanndomainsplanaranalytic}. We present one of them. Assume that $\Gamma^D$ is regular enough to provide the analyticity of the function $u$ in $\overline{\Omega}$. For example,  $\Gamma^D$ is analytic, see \cite[Remark 2.1]{anoop2024neumanndomainsplanaranalytic} for details.

\begin{definition}[\protect{\cite[remark 3.7]{anoop2024neumanndomainsplanaranalytic}}]
\label{def_nd}

The \textit{Neumann lines} of the function $u$ is given by
$$\mathcal{N}(u) := \overline{\bigcup_{r \in \mathcal{S}} \partial L(r) \setminus \Gamma^D} \cup \mathcal{C} ,$$
here boundary and closure is taken in $\mathbb{R}^2.$
The \textit{Neumann domains} of the function $u$ are connected components of $\Omega \setminus \mathcal{N}(u)$
\end{definition}

According to \cite[Theorem 3.8 (3)]{anoop2024neumanndomainsplanaranalytic}, the set of Neumann lines $\mathcal{N}(u)$ consists of a finite number of integral curves of the vector field $-\grad u$, isolated critical points, and curves of critical points.

\begin{remark}
    In \cite{anoop2024neumanndomainsplanaranalytic} the definition of the Neumann domains is given in a more general situation, namely in the case of a Dirichlet-Neumann problem. In this case we should add the part of boundary, where eigenfunction satisfy Neumann condition, to the set of Neumann lines.
\end{remark}

\begin{remark}
  Note that Neumann lines are not always a union of curves. The set $\mathcal{N}(u)$ may contain isolated critical points. For example, for the unit disk $\mathcal{N}(u_1) = \{(0, 0) \}$, so $u_1$ has exactly one Neumann domain $\{(x, y) \, | \, x^2 + y^2 \leq 1 \} \setminus \{ (0, 0) \}$.
\end{remark}

Put $G := \{ x \in \Omega : \gamma_x(t_x) \in \Gamma^D \, \text{for some} \, \,t_x \}$ and let us give one more definition.

\begin{definition}[\cite{anoop2024neumanndomainsplanaranalytic}]
    Connected components of $G \setminus \mathcal{N}(u)$ are called \textit{boundary Neumann domains}. Other connected components of $\Omega \setminus \mathcal{N}(u)$ are called \textit{interior Neumann domains}.
\end{definition}
This definition is correct due to \cite[Remark 3.7]{anoop2024neumanndomainsplanaranalytic}.

\subsection{On a motivation of the definition of the Neumann domains}

    Following \cite{MR3535866}, in this section we consider  Neumann domains in a different context. Namely, we consider the case when $u$ is a Morse function on a closed Riemannian manifold $(M, g)$. It is natural to include all critical points of $u$ in the set of the Neumann lines, because
$$\frac{\partial u}{\partial \nu}(z_0) = (\grad u (z_0), \nu) = 0$$
for any vector $\nu$, if $z_0 \in \mathcal{C}$.

  Also it is natural to include in the set of Neumann lines regular curves $\gamma(t)$ along which $\frac{\partial u}{\partial \nu}|_{\gamma} = 0$.  Such curves can be parametrized so that $\dot{\gamma}(t) = -\grad u(\gamma(t)).$ But due to the existence theorem for a system of ODE such a curve is passing through every regular point of $u$.

    Morse theory suggests how to reasonably choose a finite subset from the set of all integral curves of the vector field $-\grad u$. Assume further that $u: M \to \mathbb{R}$ is a Morse-Smale function. That means that $u$ is a Morse function and its stable and unstable manifolds intersect transversely. Let $p, \, q$ be critical points of $u$, and let their indices be $\lambda_p, \, \lambda_q$, respectively. If the set $W^u(q) \cap W^s(p)$ is nonempty, then it is a smooth immersed submanifold of dimension $\lambda_q - \lambda_p$ (see e. g. \cite[\textsection 6]{MR2145196}). For a point $r \in \mathcal{S}$ in the two-dimensional case, we have $\lambda_r = 1$, so $L(r)$ is a finite union of smooth curves. Thus, in the case when $u$ is a Morse-Smale function on a closed manifold, the Neumann lines can be defined by $\mathcal{N}(u) = L \cup \mathcal{C}$. In the paper \cite{MR3535866} it is proved that this definition also correct for just a Morse function.
  
    In the case of  manifolds with boundary, to define Neumann domains the function is extended to a Morse function on an open neighborhood of the manifold,  see \cite[\textsection 3]{MR3535866}. However, the question  when such an extension exists is non-trivial, see \cite[remark 3.3]{MR3535866}.

    Finally, let us return to the planar analytic case. Here  an eigenfunction can be analytically extended into some neighborhood of $\overline{\Omega}$. However, the definition $\mathcal{N}(u) = L \cup \mathcal{C}$ is not appropriate in this situation, because a non-Morse function  may have a \textit{trivial saddle point} (see \cite[Remark 3.4]{anoop2024neumanndomainsplanaranalytic}). This is a degenerate critical point in a neighborhood of which the function $u$ is right-equivalent to the function $\overline{u}(x, y) = \overline{u}(0) \pm x^k \pm y^2$ for odd $k$ (see \cite[Remark 2.5]{anoop2024neumanndomainsplanaranalytic}). For the function $\overline{u}$ we have $W^s(0) = \{(x, y): x\geq 0 \}$, hence $L$ contains a nonempty open subset. Therefore instead of $L(r)$, we consider the set $\partial L(r)$ in the analytic case. Also $\mathcal{C}$ and $\Gamma^N$ are naturally added to it.

\subsection{Supplementary statements from \cite{anoop2024neumanndomainsplanaranalytic}}
For convenience, we formulate several technical statements, which can be found e. g. in  \cite{anoop2024neumanndomainsplanaranalytic}. By $B_{\varepsilon}(x)$ we denote the open ball of radius $\varepsilon$ centered at the point~$x$.

\begin{proposition}[\protect{\cite[Proposition 2.24, Remark 2.25]{anoop2024neumanndomainsplanaranalytic}}]
\label{prop_224}
Let $x_0 \in \overline{\Omega}$ be a local maximum point of the eigenfunction  $u$  for problem (\ref{pr1}). Then for any $\varepsilon > 0$ there exists $\delta > 0$, such that if  $x \in B_{\delta}(x_0) \cap \overline{\Omega}$  then $\gamma_x(t) \in B_{\varepsilon}(x_0) \cap \overline{\Omega} $ for any $t \leqslant 0$, and hence $\gamma_x(- \infty) \in  B_{\varepsilon}(x_0) \cap \overline{\Omega}.$
    
\end{proposition}

\begin{lemma}[\protect{\cite[Lemma 2.28]{anoop2024neumanndomainsplanaranalytic}}]
\label{lem_228}

Let $x_0 \in \Omega $ be a regular point of  the eigenfunction  $u$  for problem (\ref{pr1}). Let $t_0$ be such that $\gamma_{x_0}(t_0) \in \Omega$. Then for any $\varepsilon > 0 $ there exists  $\delta > 0$, such that for any  $x \in B_\delta(x_0) \cap \Omega$ we have $\gamma_x(t_0) \in B_\varepsilon(\gamma_{x_0}(t_0)) \cap \Omega$.
    
\end{lemma}

\section{Supplementary statements}

\subsection{Supplementary statements for a proof of the Theorem \ref{th_uniq}} 

Let $u$ be an analytic eigenfunction of problem (\ref{pr1}). Consider a point $x \in \overline{\Omega} \setminus \mathcal{C}$ and an integral curve $\gamma_x(t)$. The ends of the curve $\gamma_x(t)$ are points of the set $\mathcal{C} \cup \Gamma^D$. Note that the left end of $\gamma_x(t)$ cannot be a local minimum, hence it belongs to the set $(\mathcal{C} \cup \Gamma^D) \setminus \mathcal{M}_-$. This follows from the fact that $u(\gamma_x(t))$ strictly monotonically increases when $t$ decreases.

For a continuous curve $\alpha: [0, 1] \to \Omega$, we define a map $E_\alpha: [0, 1] \to (\mathcal{C} \cup \Gamma^D) \setminus \mathcal{M}_-$, which maps a point $s \in [0, 1]$ to the left end of the curve $\gamma_{\alpha(s)}(t)$. It is easy to see that for  $x \in \mathcal{C}$ we have $E_\alpha^{-1}(x) = \alpha^{-1}(W^u(x))$. If the curve $\alpha$ is fixed we  write $E(s)$ instead of $E_{\alpha}(s)$.

\begin{remark}
    \label{rem_max_neq_0}

    Let $x_0$ be a local maximum of $u$, then $u(x_0) > 0$. Otherwise, there exists $\varepsilon > 0$ such that for any $x \in B_\varepsilon(x_0)$ we have $u(x) \leq 0$. But then $\Delta u \geq 0$ in $B_\varepsilon(x_0)$. So, $u$ is a subharmonic function that attains its maximum at an interior point of $B_\varepsilon(x_0)$. According to the strong maximum principle, $u(x) \equiv const$. Hence we arrive at a contradiction. 

\end{remark}
 
\begin{lemma}
\label{lemma_loc_m}
    Let $x_0 \in \overline{\Omega}$ be a point of isolated local maximum of the function $u$, then
    \begin{enumerate}
        \item the set $W^{u}(x_0) \cap \Omega$ is open,
        \item for any continuous curve $\alpha:[0,1] \rightarrow \ \Omega$, the set $E^{-1}(x_0)$ is open in $[0, \, 1]$. 
        
    \end{enumerate}
    \begin{proof}

        (1) According to the remark \ref{rem_max_neq_0} we have that  $x_0 \in \Omega$. It follows from \cite[Lemma 2.13]{anoop2024neumanndomainsplanaranalytic} that the point $x_0$ is isolated in $\mathcal{C}$.  Hence there exists $\varepsilon >0$, such that $B_\varepsilon(x_0) \subset \Omega$ and $B_\varepsilon(x_0) \cap \mathcal{C} = x_0$.   By proposition \ref{prop_224}, there exists $\varepsilon_1 $ such that for $x \in B_{\varepsilon_1}(x_0)$ we have  $\gamma_x(-\infty) =x_0$. Consider a point $y \in W^u(x_0) \cap \Omega$.  For the point $y$ there exists $t_y$, such that $\gamma_y(t_y) \in B_{\varepsilon_1}(x_0)$. Let $\varepsilon_2$ be such that $B_{\varepsilon_2}(\gamma_y(t_y)) \subset B_{\varepsilon_1}(x_0)$. If $y$ is critical, then $y = x_0$, because $y \in W^u(x_0)$. But the point $x_0$ is contained in the set $W^u(x_0) \cap \Omega$ together with its neighborhood, namely $B_{\varepsilon_1}(x_0)$. So, it remains to consider the case when $y$ is not critical. By Lemma \ref{lem_228}, there exists $\delta > 0$ such that for any $ x \in B_{\delta}(y) $ we have $\gamma_x(t_y) \in B_{\varepsilon_2}(\gamma_y(t_y))$. But then $\gamma_x(-\infty) = x_0$ for all $ x \in B_{\delta}(y)$, hence $B_{\delta}(y) \subset W^u(y)
\cap \Omega$. Therefore, $W^{u}(x_0) \cap \Omega$ is open.

(2) Under our assumption, a set $\alpha^{-1}(W^{u}(x_0) \cap \partial\Omega)$ is empty. Thus,
$$E^{-1}(x_0) = \alpha^{-1}(W^{u}(x_0) \cap \Omega)\cup \alpha^{-1}(W^{u}(x_0) \cap \partial\Omega) = \alpha^{-1}(W^{u}(x_0) \cap \Omega).$$
The curve  $\alpha$ is continuous, hence  from the first part of lemma \ref{lemma_loc_m} it follows that $ \alpha^{-1}(W^{u}(x_0) \cap \Omega)$ is open in $[0, 1]$.

    \end{proof}
\end{lemma}

\begin{lemma}
\label{lemma_loc_ni_m}
Let $\theta$ be a curve of critical points, then
\begin{enumerate}
\item the set $W^u(\theta) \cap \Omega$ is open, where $W^u(\theta) = \displaystyle{\bigcup_{x \in \theta}} W^u(x)$,
\item for any continuous curve $\alpha:[0,1] \rightarrow \Omega$ the set $E^{-1}(\theta)$ is open in $[0, \, 1]$.
\end{enumerate}

    \begin{proof}
       
            (1) From \cite[Lemma 2.13]{anoop2024neumanndomainsplanaranalytic} we have that the function $u$ does not vanish on $\theta$. Hence $\theta \subset \Omega$. Consider a point $y \in W^u(\theta) \cap \Omega$. Let $\gamma_y(- \infty) = x_0 \in \theta$. According to \cite[Lemma 2.13]{anoop2024neumanndomainsplanaranalytic} the curve $\theta$ is isolated in $\mathcal{C}$, so there exists   $\varepsilon > 0$ such that $\mathcal{C} \cap \overline{B_\varepsilon(x_0)} \subset \theta$ and $\overline{B_\varepsilon(x_0)} \subset \Omega$.  It follows from proposition \ref{prop_224} that there exists $\varepsilon_1$ such that for all $x \in B_{\varepsilon_1}(x_0)$ we have $\gamma_x(- \infty) \subset \overline{B_\varepsilon(x_0)}$. Note that ends of any integral curve lie in $\mathcal{C} \cup \Gamma^D$, so $\gamma_x(-\infty) \in (\mathcal{C} \cap \overline{B_\varepsilon(x_0)}) \subset \theta$.  For the point $y$ consider $t_y$ such that $\gamma_y(t_y) \in B_{\varepsilon_1}(x_0)$. Take $\varepsilon_2$ such that $B_{\varepsilon_2}(\gamma_y(t_y)) \subset B_{\varepsilon_1}(x_0)$. If the point  $y$ is critical, then $y = x_0\in \theta$. But $x_0$ lie in $W^{u}(\theta) \cap \Omega$ together with its neighborhood, namely with $B_{\varepsilon_1}(x_0)$.  It remains to consider the case when the point $y$ is not critical. From Lemma \ref{lem_228} we know that there exists $\delta > 0$ such that for all $x \in B_{\delta}(y)$ we have $\gamma_{x}(t_y) \in B_{\varepsilon_2}(x_0)$. Hence $\gamma_{x}(- \infty) \in \theta$. Thus, $B_{\delta}(y) \subset W^u(\theta) \cap \Omega$, so $W^u(\theta) \cap \Omega$  is open. 

             (2)  Under our assumption a set $\alpha^{-1}(W^{u}(\theta) \cap \partial\Omega)$ is empty. Hence
        $$E^{-1}(\theta) = \alpha^{-1}(W^{u}(\theta) \cap \Omega)\cup \alpha^{-1}(W^{u}(\theta) \cap \partial\Omega) =  \alpha^{-1}(W^{u}(\theta) \cap \Omega).$$
   The curve  $\alpha$ is continuous, hence  from the first part of lemma \ref{lemma_loc_ni_m} it follows that $ \alpha^{-1}(W^{u}(\theta) \cap \Omega)$ is open in $[0, 1]$.
    \end{proof}
\end{lemma}

\begin{lemma}
    \label{lemma_sadle_bound}
Let $\Omega_i$ be a Neumann domain, and let $\alpha:[0, 1] \to \Omega_i$ be a continuous curve. Suppose that $E(0) \in \mathcal{M}_+ \cap \overline{\Omega}_i $. Let $s_0 >0$ be such that for $s \in [0, s_0)$ we have $E(s) \in \mathcal{M}_+ $. Then $E(s_0) \notin \Gamma^D \setminus \mathcal{C}$ and $E(s_0) \notin \mathcal{S}$.

    \begin{proof}
Let us prove that $E(s_0) \notin \Gamma^D \setminus \mathcal{C}$ by contradiction. Let $E(s_0) \in \Gamma^D \setminus \mathcal{C}$, then there exists $t_0 < 0$ such that $\gamma_{\alpha(s_0)}(t_0) \in \Gamma^D$. Note that in this case $\Omega_i$ cannot be an interior Neumann domain. Indeed, by the definition of $G$, we have $\alpha(s_0) \in G$. Moreover, $\alpha(s_0) \notin \mathcal{N}(u)$, since $\alpha(s_0)$ belongs to a Neumann domain. But then $\alpha(s_0) \in G \setminus \mathcal{N}(u)$, so $\alpha(s_0)$ belongs to a boundary Neumann  domain.
             
If $\Omega_i$ is a boundary Neumann  domain, then by \cite[Theorem 3.8]{anoop2024neumanndomainsplanaranalytic} the function $u$ is  strictly positive or strictly negative in $\Omega_i$. Note that $\overline{\Omega}_i$ contains a maximum, so from Remark \ref{rem_max_neq_0} we conclude that the function $u$ is positive in $\Omega_i$. In particular, $u(\alpha(s_0)) > 0$. But then $u(\gamma_{\alpha(s_0)}(t_0)) >0$ due to the monotonicity of $u(\gamma_x(t))$ with respect to $t$. Thus $\gamma_{\alpha(s_0)}(t_0) \notin \Gamma^D $, so we arrive at a contradiction.

       Let us prove that $E(s_0) \notin \mathcal{S}$. Otherwise $\alpha(s_0) \in W^u(r) \subset L(r)$, for some $r \in S$. Let us show that $\alpha(s_0) \in \partial L(r)$. On the one hand, in any neighborhood of $\alpha(s_0)$ there exist points from $L(r)$, for example, points of the curve $\gamma_{\alpha(s_0)}(t)$. On the other hand, in any neighborhood of $\alpha(s_0)$ there exist points not from $L(r)$. Indeed, let us show that the points of the curve $\alpha(s)$ for $s\in (s_0 - \varepsilon, s_0)$ and sufficiently small $\varepsilon > 0$ do not belong to $L(r)$.
        
        Otherwise there exists a sequence $\{ s_k \}_{k \in \mathbb{N}}$ such that $s_k < s_0$, $\lim\limits_{k \to \infty} s_k = s_0$, and also $\alpha(s_k) \in L(r)$. Consider the curves $\gamma_{\alpha(s_k)}(t)$. Since $s_k <s_0$, we have $E(s_k) \in \mathcal{M}_+$. But $\alpha(s_k) \in L(r)$, so one of the ends of $\gamma_{\alpha(s_k)}(t)$ is $r$. Consequently, the right end of $\gamma_{\alpha(s_k)}$ is $r$.
By assumption, $\alpha(s_k) \in \Omega_i$, therefore $\alpha(s_k) \notin \mathcal{C}$. Then, from the strict monotonicity of $u$ along $\gamma_{\alpha(s_k)}$, we conclude that $u(\alpha(s_k)) > u(r)$. Passing to the limit as $k \to \infty$, we obtain $u(\alpha(s_0)) \geqslant u(r)$. Now consider the curve $\gamma_{\alpha(s_0)}(t)$. Its left end is the point $r$. Similarly, from monotonicity, we conclude that $u(r) > u(\alpha(s_0))$. Thus we have arrived at a contradiction. So, $\alpha(s_0) \in \partial L(r)$.  Since $\alpha(s_0) \notin \Gamma^D$, we conclude that $\alpha(s_0) \in \mathcal{N}(u)$. This cannot be the case, because $\alpha(s_0)$ belongs to the Neumann domain.

    \end{proof}
    
\end{lemma}

\begin{lemma}
\label{lem_final}
Let $\Omega_i$ be a Neumann domain, and let $\alpha:[0, 1] \to \Omega_i$ be a continuous curve. Then
\begin{enumerate}[(a)]
\item if $E(0) = x_0 \in \overline{\Omega}_i$ is a point of an isolated local maximum, then $E(1) = x_0$,
\item if $E(0) = x_0 \in \theta \cap \overline{\Omega}_i$, where $\theta$ is a curve of local maximums, then $E(1) \in \theta$.
    \end{enumerate}

    \begin{proof}
       
        (a) Let
$$s_0 = \sup\{s \in [0,1] : \forall t \in [0, s) \text{ we have } E(t) = x_0 \}.$$
Note that $s_0 > 0$, since $0 \in E^{-1}(x_0)$ and the set $E^{-1}(x_0)$ is open by Lemma \ref{lemma_loc_m}.

Let us show that $E(s_0) = x_0$. Otherwise, $E(s_0) \in (\mathcal{C} \cup \Gamma^D) \setminus (\{ x_0\} \cup \mathcal{M}_-)$. By Lemma \ref{lemma_sadle_bound}, we have $E(s_0) \notin \Gamma^D \setminus \mathcal{C}$ and $E(s_0) \notin \mathcal{S}$. Thus, $E(s_0) \in \mathcal{M}_+ \setminus\{ x_0\}$. If $E(s_0) = x_1$, where $x_1$ is an isolated maximum, then by Lemma \ref{lemma_loc_m} there exists $\varepsilon >0$ such that $E((s_0 -\varepsilon, s_0 ]) = x_1$. But we chose $s_0$ such that $E((s_0 -\varepsilon, s_0)) = x_0$, so this is impossible.
If $x_1$ is a non-isolated maximum, then $x_1 \in \theta_1$, where $\theta_1$ is a curve of local maxima. By Lemma \ref{lemma_loc_ni_m} there exists $\varepsilon$ such that $E((s_0 -\varepsilon, s_0]) \subset \theta_1$. This  contradicts the choice of $s_0$, since $x_0 \notin \theta_1$.

  Let us prove that $s_0 =1$. Assume the contrary, let $s_0 < 1$. We know $s_0 \in E^{-1}(x_0)$. By Lemma \ref{lemma_loc_m} the set $E^{-1}(x_0)$ is open, so for sufficiently small $\varepsilon > 0$ we have $(s_0 -\varepsilon, s_0 + \varepsilon) \subset E^{-1}(x_0)$. But this means that $E(s) = x_0$ for all $s \in [0, s_0 + \varepsilon)$. This contradicts the maximality of $s_0$. Thus, $s_0 = 1$. But it was also shown that $E(s_0) = x_0$, hence $E(1) = x_0$.

          (b) The proof of this statement is similar to the proof of previous one. Let
$$s_0 = \sup\{s \in [0,1] : \forall t \in [0, s) \text{ we have } E(t) \in \theta \}.$$

From the fact that $0 \in E^{-1}(\theta)$ and Lemma \ref{lemma_loc_ni_m}, we conclude that $s_0 > 0$. Let us  prove that $E(s_0) \in \theta$. Otherwise, $E(s_0) \in (\mathcal{C} \cup \Gamma^D) \setminus (\theta \cup \mathcal{M}_-)$. From Lemma \ref{lemma_sadle_bound}, we conclude that $E(s_0) \in \mathcal{M}_+ \setminus \theta$. Let $E(s_0) = x_1$. If $x_1$ is an isolated maximum point, then $x_1 \notin \theta$. By Lemma \ref{lemma_loc_m}, there exists $\varepsilon >0$ such that $E((s_0 -\varepsilon, s_0 ]) = x_1$. But we chose $s_0$ such that $E((s_0 -\varepsilon, s_0)) \subset \theta$, so this is impossible. If $x_1$ is a non-isolated maximum, then $x_1 \in \theta_1$, where $\theta_1$ is a curve of critical points. By Lemma \ref{lemma_loc_ni_m}, there exists $\varepsilon$ such that $E((s_0 -\varepsilon, s_0]) \subset \theta_1$. But $E((s_0 -\varepsilon, s_0 )) \subset \theta$, therefore $\theta \cap \theta_1 \neq \emptyset$.  Then $\theta = \theta_1$, and $E(s_0) \in \theta$.

Let us show that $s_0 =1$. Assume the contrary, let $s_0 < 1$. We know that $s_0 \in E^{-1}(\theta)$.  By Lemma \ref{lemma_loc_ni_m} the set $E^{-1}(\theta)$ is open, so for sufficiently small $\varepsilon > 0$ we have $(s_0 -\varepsilon, s_0 + \varepsilon) \subset E^{-1}(\theta)$. But this means that $E(s) \in \theta$ for all $s \in [0, s_0 + \varepsilon)$. This contradicts the maximality of $s_0$. Thus we have $s_0 = 1$. It was also shown that $E(s_0) \in \theta$, so $E(1) \in \theta$.

    \end{proof}
\end{lemma}

\subsection{Supplementary statements for a proof of the Theorem \ref{th2}}  Let a function $u$ be an eigenfunction of the problem (\ref{pr1}).

\begin{definition}
\label{def_nod_line}
 The \textit{nodal set} of the function $u$ is given by
$$\mathcal{D}(u) = \overline{\{ x \in \Omega: u(x) = 0\}}.$$
The connected components of $\Omega \setminus \mathcal{D}(u) $ are called the \textit{nodal domains} of $u$.
\end{definition}

The number of the nodal domains of a function $u_2$ is equal to 2, see \cite[Corollary 4.1.34]{LMP}. By $\Omega^+$ we  denote the nodal domain where $u_2$ is positive, and by $\Omega^-$ the one where $u_2$ is negative.

Payne's conjecture states that $\mathcal{D}(u_2)$ intersects $\partial \Omega$ in exactly two points. This conjecture was proved for a general convex domain \cite{MR1259610} and disproved in the non-convex case \cite{MR1480548}. For more information on Payne's conjecture, see \cite{MR1480548}.

\begin{remark}
\label{pq_crit}
     For a convex domain $\Omega$, we denote by $p, \, q$ the intersection points of $\mathcal{D}(u_2)$ with $\partial \Omega$. Note that the points $p, \, q$ are critical points. This follows from the fact that the set $u^{-1}(0)$ in the neighborhood of these points is not a smooth curve, since it consists of points $\partial \Omega \cup \mathcal{D}(u_2)$.
\end{remark}

Let us formulate the following corollary of Paine's conjecture.

\begin{lemma}
\label{lem_boud_pt}
    Let $\Omega$ be a convex domain and $u_2$ be a second eigenfunction of the problem (\ref{pr1}). Then there exist $\varepsilon > 0$ and $x^+,\, x^- \in \partial \Omega $ such that for any $x \in B_{\varepsilon}(x^+)\cap \Omega$ we have $u_2(x) > 0$, and for any $x \in B_{\varepsilon}(x^-)\cap \Omega$ we have $u_2(x) < 0$.
\end{lemma}

\begin{proof}

From definition of $\Omega^\pm$, we have $u_2|_{\partial\Omega^\pm} = 0$. Therefore, $\partial\Omega^\pm \subset \partial \Omega \cup \mathcal{D}(u_2)$. But then $\partial \Omega^\pm \setminus \mathcal{D}(u_2) \subset \partial \Omega$. Note that the sets $\partial \Omega^\pm \setminus \mathcal{D}(u_2)$ are not empty. Indeed, suppose  the contrary, let $\partial \Omega^+ \setminus \mathcal{D}(u_2) = \emptyset$. Then $\partial \Omega^+ \subset \mathcal{D}(u_2)$. But this contradicts Payne's conjecture, because according to \cite{MR1259610}, the set $\mathcal{D}(u_2)$ is a simple non-closed curve, so $\mathbb{R}^2 \setminus \mathcal{D}(u_2)$ has exactly one connected component. But then $\mathbb{R}^2 \setminus \partial \Omega^+$ has one connected component, which is impossible. Similarly, one can show that $\partial \Omega^- \setminus \mathcal{D}(u)$ is not empty.

Let $x^\pm \in \partial \Omega^\pm \setminus \mathcal{D}(u_2)$ and let $\varepsilon = \frac{1}{2} \min \{dist(x^+, \mathcal{D}(u_2)), dist(x^-, \mathcal{D}(u_2)) \}$. It is easy to see that $\varepsilon >0$. Note that $u_2$ is strictly positive or strictly negative in $B_\varepsilon(x^\pm) \cap \Omega$. Let us notice that $u_2$ is positive in $B_\varepsilon(x^+) \cap \Omega$. Indeed,  $x^+ \in \partial \Omega^+$, so there exists $x^+_0 \in \Omega^+ \cap B_\varepsilon(x^+)$. Then we have that $u_2(x^+_0) >0 $. Therefore in $B_\varepsilon(x^+) \cap \Omega$ the function is positive. Similarly, in $B_\varepsilon(x^-) \cap \Omega$ the function is negative.
\end{proof}

\begin{remark}
\label{rem_NWD}
     The set $ \mathcal{N}(u)$ is nowhere dense. Indeed, from \cite[Lemma 2.13]{anoop2024neumanndomainsplanaranalytic} we know that non-isolated critical points form smooth curves. By \cite[Theorem 3.8]{anoop2024neumanndomainsplanaranalytic}, the set $\mathcal{N}(u)$ is a finite union of isolated critical points, integral curves of the vector field $-\grad u$, and curves of critical points. Both types of curves are smooth, hence they are nowhere dense sets. Therefore, the set $\mathcal{N}(u)$ is nowhere dense, as  a finite union of nowhere dense sets.
\end{remark}

\begin{remark}
\label{rem_two_extr}
Note that for a convex domain $\Omega$, the function $u_2$ has at least four critical points in $\overline{\Omega}$. Indeed, in the nodal domain $\Omega^+$ there is at least one local maximum point. Let us denote it by $max$. Similarly in the nodal domain  $\Omega^-$ there is at least one local minimum point.  We denote it by $min$.  Also, according to remark \ref{pq_crit}, points of a set $\{ p, \, q \} = \mathcal{D}(u_2) \cap \partial \Omega$ are critical. 
\end{remark}

\begin{lemma}
    \label{lem_loc}
    Let $\Omega$ be a convex domain and let $u_2$ be a second eigenfunction of the problem (\ref{pr1}). Suppose that $u_2$ have exactly four critical points in $\overline{\Omega}$, namely, $max, \, min, \, p, \, q$. Then $\bigcup\limits_{r \in \mathcal{S}} \partial L(r) \setminus \Gamma^D $ consists of at most four integral curves, each of which connects an extremum with a saddle point. Moreover, the points $max$ and $min$ are the ends of at most two integral curves.

     \begin{proof}
By \cite[Proposition 4.13]{anoop2024neumanndomainsplanaranalytic}, the set $\bigcup \limits_{r \in \mathcal{S}} \partial L(r) \setminus \Gamma^D$ consists of integral curves together with their endpoints outside $\Gamma^D$ and curves of critical points. By our assumption, there are no curves of critical points. Consider an arbitrary integral curve $\gamma \subset \bigcup\limits_{r \in \mathcal{S}} \partial L(r) \setminus \Gamma^D$. Note that $\gamma \subset \partial L(p)$ or $\gamma \subset \partial L(q)$, this follows from \cite[Lemma 4.6]{anoop2024neumanndomainsplanaranalytic}. By \cite[Proposition 4.10]{anoop2024neumanndomainsplanaranalytic}, one end of $\gamma$ is a saddle point, hence it a point $p$ or $q$. Since $p, \, q \in \Gamma^D$, the other end of $\gamma$ cannot be on $\Gamma^D$. Therefore, the second end lies inside of $\Omega$. So it is a point $max$ or $min$, because there are no other critical points in $\Omega$.

Each curve in $\bigcup\limits_{r \in \mathcal{S}} \partial L(r) \setminus \Gamma^D$ connects a point from the set $\{ max, \, min \}$ with a point from the set $\{p, \, q \}$. Therefore, there are only four possibilities for the ends of the curve $\gamma$. Namely, the ends of $\gamma$ are points from the following sets, $\{max, \, p \}$, $\{max, \, q \}$, $\{min, \, p \}$, or $\{min, \, q \}$. Therefore, if $\bigcup\limits_{r \in \mathcal{S}} \partial L(r) \setminus \Gamma^D$ consists of more than four curves, then there exist two curves $\gamma_1, \, \gamma_2$ with the same ends. Let us consider the case when the ends of $\gamma_1, \, \gamma_2$ are  $max, \, p$. The remaining cases one can consider in a similar way.

The set $\overline{\gamma_1 \sqcup \gamma_2}$ is a Jordan curve. Indeed, the curves $\gamma_1$ and $\gamma_2$ can be reparameterized after adding ends into curves $\hat{\gamma}_1, \, \hat{\gamma}_2$ such that $\hat{\gamma}_i: [0, 1] \to \overline{\Omega}$ and $\hat{\gamma}_i(0) = p, \, \hat{\gamma}_i(1) = max$, where $i=1,\,2$. Moreover, the points $p, \, max$ are the only points of their intersection. Indeed, the original curves are integral curves of the vector field, therefore they do not intersect.

Let $\hat{\Omega}$ be the domain bounded by $\overline{\gamma_1 \sqcup \gamma_2}$. Let us show that $\hat{\Omega} \subset \Omega^+$ by contradiction. Suppose that there exists a point $x_1 \in \hat{\Omega}$ such that $u_2(x_1) \leq 0$. Since $max \in \partial \hat{\Omega}$, there exists a point in the domain $\hat{\Omega}$ where $u_2$ takes a positive value. The domain $\hat{\Omega}$ is linearly connected, hence there exists a point where $u_2$ equals $0$. We denote this point by $x_0$. Note that $q \notin \hat{\Omega}$. This follows from the fact that the set $\hat{\Omega}$ is open and lies in $\overline{\Omega}$, and $q \in \partial \Omega$. From \cite{MR1259610} we know that $\mathcal{D}(u_2)$ is a simple curve $d(t)$, where $d(0) = p, \, d(1) = q$. Since $u_2(x_0) = 0$, there exists $t_0 > 0$ such that $d(t_0) = x_0$. We know that $d(t_0) \in \hat{\Omega}$ and $d(1) = q \notin \hat{\Omega}$, hence there exists $t_1 > t_0 > 0$ such that $d(t_1) \in \partial \hat{\Omega} = \overline{\gamma_1 \sqcup \gamma_2} $. Note that the function $u_2$ on the set $\partial \hat{\Omega} = \overline{\gamma_1 \sqcup \gamma_2} $ is non-negative and takes the value zero only at the point $p$. This follows from the fact that $u_2$ is strictly monotone along $\gamma_1, \, \gamma_2$ and that $u_2(\hat{\gamma_i}(0)) = u_2(p) = 0$, $u_2(\hat{\gamma}_i(1)) = u_2(max) > 0$. But then $d(t_1) = p$, because $u_2(d(t_1)) = 0$. This contradicts the fact that $d(t)$ is a simple curve. Therefore $\hat{\Omega} \subset \Omega^+$.
             
By the Remark \ref{rem_NWD}, there exists $x \in \hat{\Omega} \setminus \mathcal{N}(u_2)$. Denote the Neumann domain that contains $x$ by $\Omega_1$. Since $\gamma_1, \, \gamma_2 \subset \partial L(p)$, we have that $\partial \hat{\Omega} \subset \mathcal{N}(u_2)$. The Neumann domains are connected components of $\Omega \setminus \mathcal{N}(u_2)$, therefore $\Omega_1 \subset \hat{\Omega}$. Consequently, $\Omega_1 \subset \Omega^+$, so by \cite[Theorem 3.8]{anoop2024neumanndomainsplanaranalytic}, the domain $\Omega_1$ is a boundary Neumann domain. By the definition of a boundary Neumann domain, there exists $t_x$ such that $\gamma_x(t_x) \in \Gamma^D$. Note that $\gamma_x(t_x) \notin \hat{\Omega}$, so there exists $t_0 \in [0, t_x]$ such that $\gamma_x(t_0) \in \partial \hat{\Omega} = \gamma_1 \sqcup \gamma_2 \sqcup \{ max, \, p\}$. Since the points $max, \, p$ are critical, $\gamma_x(t_0) \notin \{ max, p\}$. But then $\gamma_x(t_0) \in \gamma_1 \sqcup \gamma_2$, which contradicts with the uniqueness of the integral curve.

     \end{proof}

\end{lemma}

\section{Proof of the Theorems}
\label{sect_proof_th}

\subsection{Proof of the Theorem \ref{th_uniq}}

\begin{proof}

    Let us prove this Theorem by contradiction.  Let $x_0 \neq x_1 \in \overline{\Omega}_i$ be local maxima such that they does both lie on the same curve of critical points.

 Let $x_0, \, x_1$ be isolated local maxima. It follows from Proposition \ref{prop_224} that there exist points $y_0, \, y_1 \in \Omega_i$ such that $\gamma_{y_i }(-\infty) = x_j, \, j = 0, \, 1.$ Indeed,  take $\varepsilon$ such that $B_\varepsilon(x_j) \subset \Omega$, and $B_\varepsilon(x_j)$ does not contain other critical points except for $x_j, \, j = 0, \, 1.$ According to Proposition \ref{prop_224}, there exists $\delta$ such that for any $y'_j \in B_\delta(x_j)$ we have $\gamma_{y'_j }(-\infty) = x_j, \, j = 0, \, 1$. Note that the set $B_\delta(x_j) \cap \Omega_i$ is nonempty, so there exist $y_j \in \Omega_i$  such that $\gamma_{y_i }(-\infty) = x_j, \, j = 0, \, 1.$ Let us connect the points $y_0, \, y_1$ with a continuous curve $\alpha \subset \Omega_i$. Since $\alpha(0) = y_0$, we have $E_\alpha(0) = x_0$. It follows from lemma \ref{lem_final} that $E_\alpha(1) = x_0$. But then $\gamma_{y_1 }(-\infty) = x_0$, hence $x_0 = x_1$. So, we arrive at a contradiction.

Let exactly one of the maxima be isolated. Without loss of generality we assume that $x_0$ is the isolated one. Then $x_1$ lies on a curve $\theta_1$  of local maxima. Arguing similarly, we obtain that there exists a point $y_0 \in \Omega_i$ such that $\gamma_{y_0 }(-\infty) = x_0$. Let us prove that there exists a point $ y_1 \in \Omega_i$ such that $ \gamma_{y_1}(-\infty) \in \theta_1$. By \cite[Lemma 2.13]{anoop2024neumanndomainsplanaranalytic}, the curve $\theta_1$ is isolated in $\mathcal{C}$. Therefore, there exists $\varepsilon$ such that $B_\varepsilon(x_1) \cap \mathcal{C} \subset \theta_1$ and $B_\varepsilon(x_1) \subset \Omega$. By Proposition \ref{prop_224}, there exists $\delta$ such that for any $y'_1 \in B_\delta(x_1)$ we have $\gamma_{y'_1}(-\infty) \in B_\varepsilon(x_1)$. But then $\gamma_{y'_1}(-\infty) \in \theta_1 $, since $\gamma_{y'_1}(-\infty) \in \mathcal{C}$. Note that the set $B_\delta(x_1) \cap \Omega_i$ is nonempty, so there exists $y_1$ with the required property. Let us connect the points $y_0, \, y_1$ with a continuous curve $\alpha \subset \Omega_i$. Then $E_\alpha(0) = x_0$, so it follows from the Lemma \ref{lem_final} that $E_\alpha(1) = x_0$. But then $x_0 \in \theta_1$. This contradicts the assumption that $x_0, \, x_1$ do not lie on the same curve  of local maxima.

Let $x_0, \, x_1$ be non-isolated local maxima. Then they lie on curves $\theta_0, \, \theta_1$ of local maxima. Similarly, one can show that there exist points $y_0, \, y_1 \in \Omega_i$ such that $\gamma_{y_0 }(-\infty)\in \theta_0, \, \gamma_{y_1}(-\infty) \in \theta_1$. Again connect the points $y_0, \, y_1$ with a continuous curve $\alpha \subset \Omega_i$. Since $E_\alpha(0) \in \theta_0$, it follows from Lemma \ref{lem_final} that $E_\alpha(1) \in \theta_0$. But then $\theta_0 \cap \theta_1 \neq \emptyset$, which means $\theta_0 = \theta_1$. So, we arrive at a contradiction with our assumption.

\end{proof}

\subsection{Proof of the Corollary  \ref{col_1}}
\begin{proof}

Let us  define a map from the set of all isolated maxima and all curves of local maxima  to the set of Neumann domains.  Each isolated local maximum $x_0$ or  curve of local maxima  $\theta$ we map to an arbitrary Neumann domain $\Omega_i$ such that $x_0 \in \overline{\Omega}_i$ or $\theta \cap \overline{\Omega}_i \neq \emptyset$. By Theorem \ref{th_uniq}, this map is injective, which implies the required.  
    
\end{proof}

\subsection{Proof of the Corollary \ref{col_2}}
\begin{proof}

After multiplying $u$ by $-1$ if necessary, we can assume that the number of nodal domains where the function $u$ is positive is greater or equal than the number of nodal domains where $u$ is negative. Now consider an arbitrary nodal domain $\mathcal{A}_i$ such that $u$ is positive in it. Let  $x_0 \in \overline{\mathcal{A}_i}$ be a local maximum of the function $u$ in $ \overline{\mathcal{A}_i}$. Then the point $x_0$ is  isolated local maximum or $x_0 \in \theta$, where $\theta$ is a curve  of local maxima. Now we define a map from the set of nodal domains, where $u$ is positive to the set of isolated maxima and curves of maxima. Let us map the domain $\mathcal{A}_i$ to the point $x_0$ if $x_0$ is an isolated local maximum, or to the curve of critical points $\theta$ if $x_0 \in \theta$. It is easy to see that $x_0 \in \mathcal{A}_i$, $\theta \subset \mathcal{A}_i$. So, this map is injective. Therefore, the number of nodal domains where the function is positive  is less or equal than the sum of the number of isolated local maxima and the number of curves of local maxima. Therefore, by Corollary \ref{col_1} the number of nodal domains where the function is positive  is less or equal than the number of Neumann domains. Also note that, by our assumption, the number of nodal domains where $u$ is positive is greater or equal than half the number of all nodal domains.
\end{proof}

\subsection{Proof of the Theorem \ref{th1}}

Let us introduce some notation from \cite{ASNSP_1992_4_19_4_567_0}. Let $u$ be an eigenfunction of the problem (\ref{pr1}). By $n_S$ we denote the number of saddle points of the function $u$ at which it is non-zero. By $n_E$ we denote the number of isolated extrema. Note that in \cite{ASNSP_1992_4_19_4_567_0}, only points of the set $\mathcal{S} \setminus \mathcal{C}_{sing}$ are called saddle, but we  continue to use the terminology that was introduced in the Section \ref{sec_def}. Let us prove the Theorem \ref{th1}.

\begin{proof}

The first eigenfunction of the Dirichlet problem is positive or negative in the domain $\Omega$. Without loss of generality we can assume that the function $u_1$ is positive. From \cite[Theorem 1]{article} it follows that for a domain $\Omega$ with a boundary of positive curvature, the function $u_1$ has a unique critical point $x_0$ in $\Omega$, and that the point $x_0$ is a point of local maximum.
    
   Let us check that the function $u_1$ has no other critical points in $\overline{\Omega}$. According to \cite[ Corollary 3.4]{ASNSP_1992_4_19_4_567_0}, for a domain with an analytic boundary, we have
$$n_S - n_E = -1.$$ Hence $n_S = 0$.
Note that in \cite{ASNSP_1992_4_19_4_567_0} it is not specified that these are the numbers of critical points in $\overline{\Omega}$ and not in $\Omega$. But in proof of \cite[ Corollary 3.4]{ASNSP_1992_4_19_4_567_0} it is shown that the function $u_1$ has no critical points on $\partial \Omega$.

Now let us prove that the function $u_1$ has exactly one Neumann domain. Since $\mathcal{S} = \emptyset, \mathcal{M}_- = \emptyset $ and $\mathcal{M}_+ = \{x_0\}$, we obtain that
$$\mathcal{N}(u_1) = \overline{\bigcup_{r \in \mathcal{S}} \partial L(r) \setminus \Gamma^D} \cup \mathcal{C}=\{ x_0\}.$$ Since $\Omega$ is a domain and the Neumann domains are the connected components of $\Omega \setminus \mathcal{N}(u_1) = \Omega \setminus \{ x_0 \}$, we conclude that their number is equal to one.

\end{proof}

\subsection{Proof of the statement (1) of the Theorem \ref{th2}}

\begin{proof}
Let us show that the function $u_2$ has at least two boundary Neumann domains. By Lemma \ref{lem_boud_pt}, there exist points $x^+, x^- \in \Gamma^D$ and their neighborhoods $B_\varepsilon(x^+), B_\varepsilon(x^-)$ such that $u_2(x) > 0$ for all $x \in B_\varepsilon(x^+) \cap \Omega$ and $u_2(x) < 0$ for all $x \in B_\varepsilon(x^-) \cap \Omega$. It follows from \cite[Lemma 4.1]{anoop2024neumanndomainsplanaranalytic} that $\Gamma^D \subset \partial G$, therefore $x^+, x^- \in \partial G$. Thus, there exist points $y^+, y^- \in G$ such that $y^+ \in B_\varepsilon(x^+)$ and $y^- \in B_\varepsilon(x^-)$. Since the sets $G$, $B_{\varepsilon}(x^{\pm})$ are open, there exists a sufficiently small $\delta$ such that $B_\delta(y^+) \subset G \cap B_{\varepsilon}(x^+)$ and $B_\delta(y^-) \subset G \cap B_{\varepsilon}(x^-)$. Finally, there exist points $z^+ \in B_{\delta}(y^+) \setminus \mathcal{N}(u_2)$ and $z^- \in B_{\delta}(y^-) \setminus \mathcal{N}(u_2)$, since  the set $\mathcal{N}(u_2)$ is nowhere dense by Remark \ref{rem_NWD}. By construction, $z^+, \, z^- \in G \setminus \mathcal{N}(u_2)$, hence the set of boundary Neumann domains is non-empty. Moreover, there are at least two boundary Neumann domains. Indeed, the points $z^+, z^-$ cannot belong to the same boundary Neumann domain, because by \cite[Theorem 3.8]{anoop2024neumanndomainsplanaranalytic} the function $u_2$ is sign-definite in the boundary Neumann domain.

Now we show that there is at least one interior Neumann domain. Note that the set $u_2^{-1}(0) \cap \Omega$ is infinite. Otherwise, the set $\Omega \setminus u_2^{-1}(0)$ is  a domain in which the continuous function $u_2$ takes only strictly positive and strictly negative values. That is impossible. Let us show that $u_2^{-1}(0) \cap \mathcal{N}(u_2)$ is finite. Indeed, by \cite[Theorem 3.8]{anoop2024neumanndomainsplanaranalytic} the set $\mathcal{N}(u_2)$ is a finite union of integral curves of the vector field $-\grad u_2$, isolated critical points, and curves  of critical points. The set $u_2^{-1}(0)$ intersects every integral curve at no more than one point.  According to \cite[Lemma 2.13]{anoop2024neumanndomainsplanaranalytic}, the set $u_2^{-1}(0)$ does not intersect  curves of critical points. Thus, the set $u_2^{-1}(0) \cap \mathcal{N}(u_2)$ is finite, as a finite union of finite sets. Consequently, there exists a point $x_0 \in (u_2^{-1}(0) \cap \Omega) \setminus \mathcal{N}(u_2)$. Thus $x_0$ belongs to some Neumann domain. But this domain is interior, since $u_2(x_0) = 0$ and the function $u_2$ is sign-definite in the boundary Neumann domain by \cite[theorem 3.8]{anoop2024neumanndomainsplanaranalytic}.
    
\end{proof}

\subsection{Proof of the statement (2) of the Theorem \ref{th2}}

We need to introduce  more notation from \cite{ASNSP_1992_4_19_4_567_0}. In the Euclidean plane with coordinates $x, \, y$, we introduce the complex coordinate $z = x + i y$. By $\partial_z$ we denote the complex derivative, that is, $\partial_z = \frac{1}{2}(\partial_x - i \partial_y)$. Let $\Omega$ be a convex bounded domain with analytic boundary. In this case, $\partial \Omega$ has exactly one connected component. Suppose that the function $u$ is a Dirichlet eigenfunction of $\Omega$, and all critical points of $u$ are isolated. Then, according to \cite[Theorem 3.5]{ASNSP_1992_4_19_4_567_0}, for any $z_0 \in \mathcal{C}_{sing}$, the function $\partial_z u$ is asymptotically equivalent to the function $c_0 (z - z_0)^{m_0}$, as $z \to z_0$, where $c_0 \in \mathbb{C}, \, m_0 \in \mathbb{N}$ are some constants. The number $m_0$ is called the \textit{integral multiplicity} of zero $z_0$, see \cite[Theorem 3.5]{ASNSP_1992_4_19_4_567_0}. For points $z_0 \notin \mathcal{C}_{sing}$, we set $m_0 = 0.$
 
\begin{proof}

We only need to show that the number of Neumann domains does not exceed three. Recall that the intersection points of $\mathcal{D}(u)$ with $\partial \Omega$ are denoted by $p ,\, q$. By Remark \ref{pq_crit}, the points $p, q$ belong to $C_{sing}$. Then, from \cite[Theorem 3.5]{MR1259610} it follows that

$$\sum_{x_k \in \Omega } m_k + \frac{1}{2} \sum_{x_k \in \partial\Omega} m_k + n_S - n_E = -1.$$
Let us substitute $n_E = 2$ and consider separately  the terms corresponding to the points $p, \, q$ in a sum $\frac{1}{2} \sum_{x_k \in \partial\Omega} m_k$. Then we obtain

$$\sum_{x_k \in \Omega} m_k + \frac{1}{2} \sum_{x_k \in \partial\Omega \setminus \{p, \, q \}} m_k + n_S + \frac{1}{2}(m_p + m_q - 2)= 0.$$
Since points $p, \, q$ are critical, we have that $m_p \geq 1, m_q \geq 1$. Therefore, all terms on the left-hand side of the equality are nonnegative, hence each term is equal to zero. Thus, in $\overline{\Omega}$ there are exactly $4$ critical points: $p, \, q, \, max, \, min$, where $max, \, min$ are the maximum and minimum points, respectively.
   Points $p, \, q$ split $\partial \Omega$ into two simple curves. We denote them by $\Gamma_1, \, \Gamma_2$. Consider the following cases.

     (a) The points $max, \, min$ are isolated points in $\mathcal{N}(u)$. By Lemma \ref{lem_loc}, the set $\bigcup\limits_{r \in \mathcal{S}} \partial L(r) \setminus \Gamma^D$ is empty, since every curve from it has an end in the set $\{max, \, min \}$. Therefore $\mathcal{N}(u) = \{ max, min, p, q\}$, so $\Omega \setminus \mathcal{N}(u)$ has exactly one connected component. By the statement (1) of the Theorem \ref{th2}, this case is impossible.
     
     (b) The point $min$ is an isolated point in $\mathcal{N}(u)$, but the point $max$ is not. By Lemma \ref{lem_loc}, the set $\bigcup\limits_{r \in \mathcal{S}} \partial L(r) \setminus \Gamma^D$ consists of at most two integral curves $\gamma_1, \, \gamma_2$, each of which connects $max$ with a saddle point. Consider a planar graph $G$ with vertices $max, \, p, \, q$ and edges $\gamma_1, \gamma_2, \, \Gamma_1, \, \Gamma_2$, where one of the edges $\gamma_1$ or $\gamma_2$ may be absent. Note that $G \cup \{ min\} = \mathcal{N}(u) \cup \Gamma^D $, so the number of bounded faces of $G$ is equal to the number of Neumann domains. We denote by $V$ the number of vertices of $G$, by $E$ the number of edges, and by $F$ the number of faces. We apply Euler's formula $V- E + F = 2$ (see e. g. \cite[Theorem 13.1]{wilson1979introduction}) for a planar connected graph $G$. Since $E \leqslant 2$ and $V = 3$, we have that $F \leqslant 3$. Consequently, the number of bounded faces of $G$ does not exceed two. So, this case is again impossible.

      (c) Both of the points $max, \, min$ are non-isolated in $\mathcal{N}(u)$. By Lemma \ref{lem_loc}, the set $\bigcup\limits_{r \in \mathcal{S}} \partial L(r) \setminus \Gamma^D$ consists of at most four integral curves $\gamma_1, \, \gamma_2, \, \gamma_3, \, \gamma_4$ each of which connects an extremum with a saddle point. Now consider a planar graph $G$ with vertices $max, \, min, \, p, \, q$ and edges $\gamma_1, \, \gamma_2,\, \gamma_3,\, \gamma_4, \, \Gamma_1, \, \Gamma_2$, where some of the edges $\gamma_1, \, \gamma_2,\, \gamma_3,\, \gamma_4,$ may be absent. Note that $G = \mathcal{N}(u) \cup \Gamma^D$, so the number of bounded faces of $G$ is equal to the number of Neumann domains. Again from Euler's formula for $G$ we obtain the inequality $F \leqslant 4$. Hence the number of Neumann domains does not exceed three.

\end{proof}

\section{Examples}
\label{examples}

\begin{example}
The first eigenfunction $u_1 = -sin( \pi x) sin(\pi y)$ of the Dirichlet problem for the unit square has four Neumann domains (Fig. \ref{fig2}). For domains without analytical boundary, the first eigenfunction can have more than one critical point. This may result in a function having more than one Neumann domain. In this case, the function $u_1$ has four saddle points on the boundary at the vertices of the square.
\end{example}

\begin{example}
For each $n \in \mathbb{N}$ and $a \in (\frac{1}{4}, \frac{1}{2})$, we define the domain $\Omega_{n,a}$ as follows,
$$\Omega_{n,a} := \{x \in \mathbb{R}^2 : a<r(x) < (1 + \frac{1}{2} cos(n \varphi(x))) \},$$
where $(r(x), \varphi(x))$ are the polar coordinates of the point $x$ (see Fig. \ref{fig2}). The boundary of the domain $\Omega_{n,a}$ consists of two curves, a circle of radius $a$, which we denote by $\Gamma_1$, and the curve $r(\varphi) = 1 + \frac{1}{2} cos(n \varphi)$, which we denote by $\Gamma_2$. So, the boundary of $\Omega_{n,a}$ is analytic. In Lemma \ref{max_localiz} and Lemma \ref{lemma_max_sym}, we show that the parameter $ a$ can be chosen such that the function $u_1(\Omega_{n, a})$ has at least $n$ local maxima in $\Omega_{n,a}$. Moreover these maxima pairwise do not  lie on the same curve of critical points.

\end{example}

\begin{figure}
    \centering

    \centering
    \subfloat[\centering ]{{\includegraphics[width=6cm]{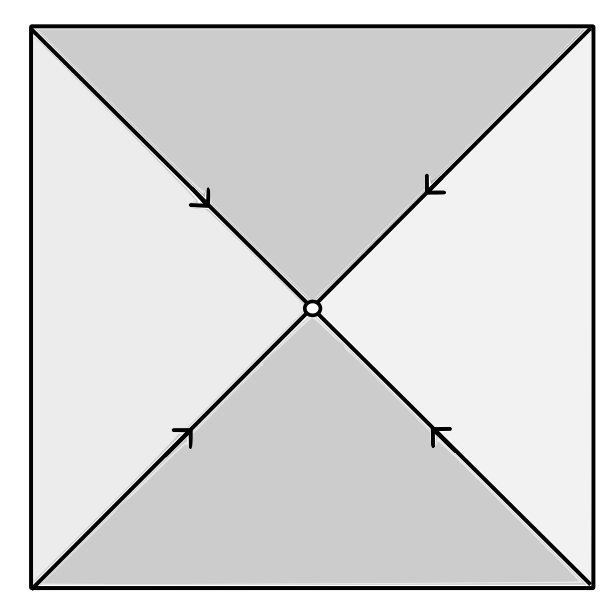} }}%
    \qquad
    \subfloat[\centering ]{{\includegraphics[width=6cm]{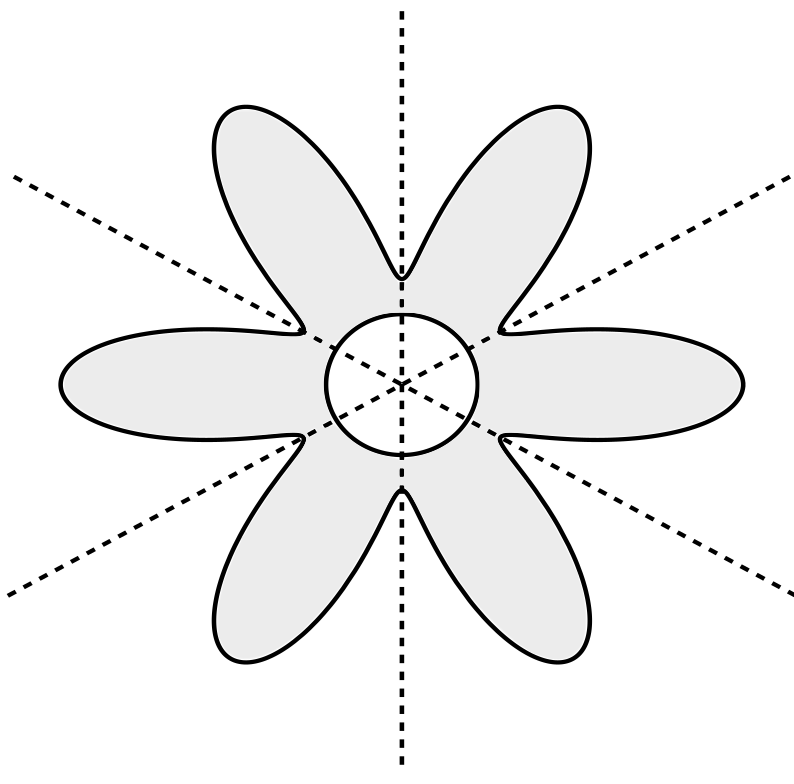} }}
    \caption{(A) Neumann domains of the first eigenfunction of a square. (B) The domain $\Omega_{6, 1/3}$ and  the set $L_6$.} 
    \label{fig2}

\end{figure}

For each $k= 1,..., n$, by $L_{n, k}$ we denote the ray with origin at $(0, 0)$ that forms an angle $\frac{\pi}{n} + \frac{2 \pi k}{n}$ with the positive direction of the $Ox$-axis. We also set $L_n = \cup_{k}L_{n, k}$. Let $x_{max}$ be the global maximum point of $u_1$ in the domain $\Omega_{n, a}$. We assume that $u_1$ is normalized such that $||u_1||_{\infty} = 1$ and $u_1(x_{max}) = 1$.

To prove Lemma \ref{max_localiz}, we need the estimate for $||\grad u_1||_{\infty}$ from the paper \cite{MR4172683}. In the notation introduced below, it is  taken into account that the Ricci curvature of the domain $\Omega_{n, a}$ is equal to zero. By $H_{\partial\Omega_{n,a}}$ we denote the mean curvature of the boundary of the domain $\Omega_{n,a}$. In this case, $H_{\partial\Omega_{n,a}}$ coincides with the curvature of a planar curve. Let $\theta$ be a non-negative constant such that $H_{\partial\Omega_{n,a}} \geq -\theta$. Let $A := \theta + \sqrt{\frac{2\lambda_1}{\pi}} \text{exp}(- \frac{\theta^2}{8 \lambda_1})$, where $\lambda_1 = \lambda_1(\Omega_{n, a})$. According to \cite[Theorem 1.1]{MR4172683}, for $\sqrt{\lambda_1} \leq 2 A $ we have
\begin{equation}
\label{norm_ine}
||\grad u_1||_{\infty} \leq \sqrt{e} (A + \frac{\lambda_1}{4 A}) ||u_1||_{\infty}.
\end{equation}

\begin{lemma}
\label{max_localiz}
   For every $n \in \mathbb{N}$ there exists $a \in (\frac{1}{4}, \frac{1}{2})$ such that for every $x \in \Omega_{n, a} \cap L_n$ we have $|u_1(x)| < 1$.
    \begin{proof}
Firstly, note that the domain $\Omega_{n,a}$ contains a ball of radius $C(n)$ for any  $a \in (\frac{1}{4}, \frac{1}{2})$. Indeed, consider the point $x_0 = (1, 0) \in \Omega_{n,a}$. We have $dist(x_0, \Gamma_1) > \frac{1}{2}$, and also $dist(x_0, \Gamma_2) > C(n)$, since $x_0 \notin \Gamma_2$. Therefore, a ball of the required radius can be found. Now we apply the monotonicity property for the Dirichlet eigenvalues (see e.g. \cite[Theorem 3.2.1]{LMP}) to the domains $B_{C(n)}(x_0) \subset\Omega_{n,a}$ and obtain $\lambda_1(\Omega_{n,a}) \leq \lambda_1(B_{C(n)}(x_0)) =: \Lambda(n)$.

There exists a constant $\theta(n)$ such that $|H_{\partial\Omega_{n,a}}| \leq \theta(n)$, for any $a \in (\frac{1}{4}, \frac{1}{2})$. Indeed, the curve $\Gamma_2$ is compact and does not depend on the parameter $a$, so $|H_{\Gamma_2}| \leq C_2(n)$, for some constant $C_2(n)$. The curve $\Gamma_1$ is a circle of radius $a \in (\frac{1}{4}, \frac{1}{2})$, so $|H_{\Gamma_1}| \leq 16$. Let $\theta(n) = \max (C_2(n), 16, \frac{1}{2}\sqrt{\Lambda(n)} )$.

   For this choice of $\theta(n)$, for any $a \in (\frac{1}{4}, \frac{1}{2})$, we have
$$ \sqrt{\lambda_1(\Omega_{n, a})} \leq \sqrt{\Lambda(n)} \leq 2 \theta(n) \leq 2 A(n, a).$$
Thus, the estimate (\ref{norm_ine}) holds. Also note that $\frac{1}{2} \sqrt{\Lambda(n)} \leq A(n, a) \leq \theta(n) + \sqrt{\Lambda(n)}$, and that $ \lambda_1(\Omega_{n, a}) \leq \Lambda(n)$ for any $a \in (\frac{1}{4}, \frac{1}{2})$.  Let $F(n)$  be  the maximum value of the function $\sqrt{e} (x + \frac{\Lambda(n)}{4 x})$ on the interval $[\frac{1}{2} \sqrt{\Lambda(n)}, \theta(n) + \sqrt{\Lambda(n)}]$. Thus, we obtain
$$ ||\grad u_1||_{\infty} \leq F(n) ||u_1||_{\infty} = F(n) .$$

Let $a$ be such that $(\frac{1}{2} - a) F(n) < 1$. Consider an arbitrary point  $x \in \Omega_{n, a} \cap L_n$, then $x \in \Omega_{n, a} \cap L_{n, k}$ for some $k$. Note that the set $\Omega_{n, a} \cap L_{n, k}$ is a segment of length $(\frac{1}{2} - a)$. Consider a point $y$ with polar coordinates $r = a, \, \varphi = \frac{\pi}{n} + \frac{2 \pi k}{n}$. Note that $y \in \partial\Omega_{n, a} \cap L_{n, k}$, hence $u_1(y) = 0$. Using
$$|u_1(x) - u_1(y)| \leq ||\grad u_1||_{\infty} \text{dist}(x, y).$$
we obtain that $|u_1(x)| \leq F(n)(\frac{1}{2} - a) < 1$.

    \end{proof}
\end{lemma}

Let $\Omega^n := \Omega_{n, a}$, where $a$ is chosen such that for any $x \in \Omega_{n, a} \cap L_n$ we have $|u_1(x)| < 1$.

\begin{lemma}
\label{lemma_max_sym}
In the domain $\Omega^n$ there exists local maxima $x_1, ..., x_n$. Moreover these maxima pairwise do not  lie on the same curve of critical points.

    \begin{proof}
Note that the domain $\Omega_{n,a}$ is invariant under rotations by $\frac{2 \pi k}{n}$. Since the first eigenvalue is simple (see, e.g., \cite[Corollary 4.1.32]{LMP}), we have that the function $u_1$ is also invariant under rotations by these angles. Put $x_k = R_{\frac{2 \pi k}{n}}(x_{max})$, where $R_\varphi$ is the rotation operator by an angle $\varphi$ about the origin. The points $x_k$ are local maxima. Note that $\Omega_{n, a} \setminus L_n$ consists of $n$ connected components $\mathcal{A}_1, ..., \mathcal{A}_n$. Using $u_1(x_k) = 1$ and the definition of $\Omega^n$, we obtain that $x_k \notin L_n$. So, after renumbering the connected components if necessary, we can assume that $x_k \in \mathcal{A}_k$.
         
If  $x_i \in \theta_i$, where $\theta_i$ is  a curve of critical points, then $\theta_i \subset \mathcal{A}_i$. Otherwise, there exists a point $x \in \theta_i \cap \partial\mathcal{A}_i$. But $\partial\mathcal{A}_i \subset \partial \Omega^n \cup L_n$, so, by construction, $|u_1(x)| < 1$. This contradicts with the fact that $u_1$ is constant on the curve $\theta_i$, see \cite[Lemma 2.13 (iii)]{anoop2024neumanndomainsplanaranalytic}, since $u_1(x_i) = 1, \, |u_1(x)| < 1$.
        \end{proof}
\end{lemma}

\section*{Acknowledgements}
The work was supported by the Theoretical Physics and Mathematics
Advancement Foundation «BASIS» Stipend (Student) 24-8-2-19-1 and 25-8-2-18-1.

The author would like to thank A. V. Penskoi for attracting attention to this problem and valuable discussions.

\bibliographystyle{alpha}
\bibliography{bibl.bib}

@article{ASNSP_1992_4_19_4_567_0,
     author = {Alessandrini, G. and Magnanini, R.},
     title = {The index of isolated critical points and solutions of elliptic equations in the plane},
     journal = {Annali della Scuola Normale Superiore di Pisa - Classe di Scienze},
     pages = {567--589},
     publisher = {Scuola normale superiore},
     volume = {Ser. 4, 19},
     number = {4},
     year = {1992},
     mrnumber = {1205884},
     zbl = {0793.35021},
     language = {en},
     url = {http://www.numdam.org/item/ASNSP_1992_4_19_4_567_0/}
}

@misc{anoop2024neumanndomainsplanaranalytic,
      title={Neumann domains of planar analytic eigenfunctions}, 
      author={T. V. Anoop and Vladimir Bobkov and Mrityunjoy Ghosh},
      year={2024},
      eprint={2410.07811},
      note={arXiv preprint  	arXiv:2410.07811},
      archivePrefix={arXiv},
      primaryClass={math.AP},
      url={https://arxiv.org/abs/2410.07811}, 
}

@article{article,
author = {Cabre, Xavier and Chanillo, Sagun},
year = {1998},
month = {03},
pages = {1-10},
title = {Stable solutions of semilinear elliptic problems in convex domains},
volume = {4},
journal = {Selecta Mathematica},
doi = {10.1007/s000290050022}
}

@article {MR1259610,
    AUTHOR = {Alessandrini, Giovanni},
     TITLE = {Nodal lines of eigenfunctions of the fixed membrane problem in
              general convex domains},
   JOURNAL = {Comment. Math. Helv.},
  FJOURNAL = {Commentarii Mathematici Helvetici},
    VOLUME = {69},
      YEAR = {1994},
    NUMBER = {1},
     PAGES = {142--154},
      ISSN = {0010-2571},
   MRCLASS = {35P05 (35J05)},
  MRNUMBER = {1259610},
MRREVIEWER = {Gianfranco Bottaro},
       DOI = {10.1007/BF02564478},
       URL = {https://doi.org/10.1007/BF02564478},
}

@article {MR3535866,
    AUTHOR = {Band, Ram and Fajman, David},
     TITLE = {Topological properties of {N}eumann domains},
   JOURNAL = {Ann. Henri Poincar\'{e}},
  FJOURNAL = {Annales Henri Poincar\'{e}. A Journal of Theoretical and
              Mathematical Physics},
    VOLUME = {17},
      YEAR = {2016},
    NUMBER = {9},
     PAGES = {2379--2407},
      ISSN = {1424-0637},
   MRCLASS = {58J50 (35A16 35P05 35R01)},
  MRNUMBER = {3535866},
MRREVIEWER = {Leonid Friedlander},
       DOI = {10.1007/s00023-016-0468-7},
       URL = {https://doi.org/10.1007/s00023-016-0468-7},
}

@book {MR2145196,
    AUTHOR = {Banyaga, Augustin and Hurtubise, David},
     TITLE = {Lectures on {M}orse homology},
    SERIES = {Kluwer Texts in the Mathematical Sciences},
    VOLUME = {29},
 PUBLISHER = {Kluwer Academic Publishers Group, Dordrecht},
      YEAR = {2004},
     PAGES = {x+324},
      ISBN = {1-4020-2695-1},
   MRCLASS = {58E05 (53D40 57R99)},
  MRNUMBER = {2145196},
MRREVIEWER = {Michael J. Usher},
       DOI = {10.1007/978-1-4020-2696-6},
       URL = {https://doi.org/10.1007/978-1-4020-2696-6},
}

@book {LMP,
    AUTHOR = {Levitin, Michael and Mangoubi, Dan and Polterovich, Iosif},
     TITLE = {Topics in spectral geometry},
    SERIES = {Graduate Studies in Mathematics},
    VOLUME = {237},
 PUBLISHER = {American Mathematical Society, Providence, RI},
      YEAR = {[2023] \copyright 2023},
     PAGES = {xviii+325},
      ISBN = {[9781470475253]; [9781470475482]; [9781470475499]},
   MRCLASS = {35-01 (35Pxx 47A75 58J50 58J53 65N30)},
  MRNUMBER = {4655924},
MRREVIEWER = {Yoonweon Lee},
       DOI = {10.1090/gsm/237},
       URL = {https://doi.org/10.1090/gsm/237},
}

@article {MR1480548,
    AUTHOR = {Hoffmann-Ostenhof, M. and Hoffmann-Ostenhof, T. and
              Nadirashvili, N.},
     TITLE = {The nodal line of the second eigenfunction of the {L}aplacian
              in {${\bf R}^2$} can be closed},
   JOURNAL = {Duke Math. J.},
  FJOURNAL = {Duke Mathematical Journal},
    VOLUME = {90},
      YEAR = {1997},
    NUMBER = {3},
     PAGES = {631--640},
      ISSN = {0012-7094},
   MRCLASS = {35P15 (35J05)},
  MRNUMBER = {1480548},
MRREVIEWER = {B. Hellwig},
       DOI = {10.1215/S0012-7094-97-09017-7},
       URL = {https://doi.org/10.1215/S0012-7094-97-09017-7},
}

@article {MR3151083,
    AUTHOR = {McDonald, Ross B. and Fulling, Stephen A.},
     TITLE = {Neumann nodal domains},
   JOURNAL = {Philos. Trans. R. Soc. Lond. Ser. A Math. Phys. Eng. Sci.},
  FJOURNAL = {Philosophical Transactions of the Royal Society of London.
              Series A. Mathematical, Physical and Engineering Sciences},
    VOLUME = {372},
      YEAR = {2014},
    NUMBER = {2007},
     PAGES = {20120505, 6},
      ISSN = {1364-503X},
   MRCLASS = {81Q10},
  MRNUMBER = {3151083},
       DOI = {10.1098/rsta.2012.0505},
       URL = {https://doi.org/10.1098/rsta.2012.0505},
}

@article {MR4668090,
    AUTHOR = {Band, Ram and Cox, Graham and Egger, Sebastian K.},
     TITLE = {Defining the spectral position of a {N}eumann domain},
   JOURNAL = {Anal. PDE},
  FJOURNAL = {Analysis \& PDE},
    VOLUME = {16},
      YEAR = {2023},
    NUMBER = {9},
     PAGES = {2147--2171},
      ISSN = {2157-5045},
   MRCLASS = {58J50 (35J05 35P05 35R01 37D10 58C40)},
  MRNUMBER = {4668090},
MRREVIEWER = {G\"{u}nter Berger},
       DOI = {10.2140/apde.2023.16.2147},
       URL = {https://doi.org/10.2140/apde.2023.16.2147},
}

@article {MR4244877,
    AUTHOR = {Band, Ram and Egger, Sebastian K. and Taylor, Alexander J.},
     TITLE = {The spectral position of {N}eumann domains on the torus},
   JOURNAL = {J. Geom. Anal.},
  FJOURNAL = {Journal of Geometric Analysis},
    VOLUME = {31},
      YEAR = {2021},
    NUMBER = {5},
     PAGES = {4561--4585},
      ISSN = {1050-6926},
   MRCLASS = {58C40 (35P05 58J50)},
  MRNUMBER = {4244877},
MRREVIEWER = {Mohammed El A\"{\i}di, Universidad Nacional de Colombia},
       DOI = {10.1007/s12220-020-00444-9},
       URL = {https://doi.org/10.1007/s12220-020-00444-9},
}

@incollection {MR3087922,
    AUTHOR = {Zelditch, Steve},
     TITLE = {Eigenfunctions and nodal sets},
 BOOKTITLE = {Surveys in differential geometry. {G}eometry and topology},
    SERIES = {Surv. Differ. Geom.},
    VOLUME = {18},
     PAGES = {237--308},
 PUBLISHER = {Int. Press, Somerville, MA},
      YEAR = {2013},
   MRCLASS = {58J50 (35P05 35R01 53Dxx)},
  MRNUMBER = {3087922},
       DOI = {10.4310/SDG.2013.v18.n1.a7},
       URL = {https://doi.org/10.4310/SDG.2013.v18.n1.a7},
}

@book{wilson1979introduction,
  title={Introduction to graph theory},
  author={Wilson, Robin J},
  year={1979},
  publisher={Pearson Education India}
}

@article{MR4172683,
  title={Gradient estimates on Dirichlet and Neumann eigenfunctions},
  author={Arnaudon, Marc and Thalmaier, Anton and Wang, Feng-Yu},
  journal={International Mathematics Research Notices},
  volume={2020},
  number={20},
  pages={7279--7305},
  year={2020},
  publisher={Oxford University Press}
}

@article{DEREGIBUS2022109496,
title = {On the number of critical points of the second eigenfunction of the Laplacian in convex planar domains},
journal = {Journal of Functional Analysis},
volume = {283},
number = {1},
pages = {109496},
year = {2022},
issn = {0022-1236},
doi = {https://doi.org/10.1016/j.jfa.2022.109496},
url = {https://www.sciencedirect.com/science/article/pii/S0022123622001161},
author = {Fabio {De Regibus} and Massimo Grossi}
}

\end{document}